\documentclass[12pt]{amsart}

\usepackage{latexsym,amssymb,amsmath,amsthm,amscd,graphicx, xcolor, enumerate}

\usepackage{times}
\usepackage{microtype}
\usepackage{setspace}
\usepackage[margin=1.2in]{geometry}
\usepackage[T2A]{fontenc}
\usepackage[utf8]{inputenc}
\usepackage[russian,english]{babel}
\usepackage{mathtools}
\usepackage{lmodern}

\usepackage{cite}

\usepackage[colorlinks=true, pdfstartview=FitV, linkcolor=blue,
citecolor=blue, urlcolor=blue]{hyperref}

\numberwithin{equation}{section}
\allowdisplaybreaks

\numberwithin{equation}{section}
\newtheorem{theorem}{Theorem}[section]
\newtheorem{lemma}[theorem]{Lemma}

\theoremstyle{definition}

\newtheorem{remark}[theorem]{Remark}
\newtheorem{definition}[theorem]{Definition}

\theoremstyle{remark}

\title[Bochkarev's inequalities in the Anisotropic grand Lorentz Spaces]{Bochkarev's inequalities in the Anisotropic grand Lorentz Spaces}

\author[N. T. Tleukhanova]{Nazerke T. Tleukhanova}
\address	{Nazerke T. Tleukhanova\\
Department of Fundamental Mathematics\\
		L.N. Gumilyov Eurasian National University\\
		13, Kazhymukan  Ave, 010000  Astana\\
		Kazakhstan \\ }
\email{tleukhanova@rambler.ru}

\author[M. Manarbek]{Makhpal Manarbek}
\address{Makhpal Manarbek\\ Department of Fundamental Mathematics\\
		L.N. Gumilyov Eurasian National University,\\
		and
		Institute of Mathematics and Mathematical Modelling, Almaty
		Kazakhstan}
\email{manarbek@math.kz}
\begin{document}

\begin{abstract}  
The main aim of this paper is to obtain Bochkarev-type inequalities for the anisotropic grand Lorentz spaces. 
In the classical setting, Bochkarev obtained inequalities of the Hardy--Littlewood type, which reveal the connection between the integral properties of functions and the summability of their Fourier coefficients. 
His results describe the behavior of trigonometric series in the Lorentz spaces $L_{2,r}$ for $2 < r \le \infty$. 
In this work, we extend these ideas to the framework of anisotropic grand Lorentz spaces. Using an approach based on the extrapolation of linear operators, we derive new Bochkarev-type inequalities that generalize the classical results to the case of anisotropic grand  Lorentz spaces and arbitrary orthonormal systems. We investigate the summability properties of the Fourier coefficients of functions from anisotropic grand Lorentz spaces.
\end{abstract}  
		\maketitle	
		\section{Introduction}
		One of the principal directions of harmonic analysis is the theory of Fourier series \cite{1,2,3,4}.
Interest in this area is driven by its applications across modern mathematics and the applied sciences,
as well as by the presence of many difficult open problems\cite{5, 6, 7, 8, 9, 10, 11, 12, 13, 14, 15, 16}.
Parseval's identity is perhaps the most familiar result of this type \cite{Bari1961}. 
Let $f$ be an integrable function on $[0,1]$, and let 
$\{\varphi_k\}_{k=1}^{\infty}$ be an orthonormal system in $L_2([0,1])$. 
The Fourier coefficients of $f$ with respect to this system are defined by
\[
a_k = a_k(f) := \int_0^1 f(x)\,\overline{\varphi_k(x)}\, dx .
\]

When $f \in L_2([0,1])$, Parseval's identity asserts that
\[
\|f\|_{L_2} = \|a\|_{\ell_2} := 
\left( \sum_{k \in \mathbb{Z}} |a_k|^2 \right)^{1/2},
\]
where
\[
a_k = \int_0^1 f(x) e^{-2\pi i k x}\, dx
\]
are the standard Fourier coefficients associated with the trigonometric 
orthonormal system.
 \\
		Historically in 1956, Stein \cite{21} proved the Hardy--Littlewood inequality for Lorentz spaces, which generalizes Paley’s inequality:
		
		\begin{enumerate}
			\item If $1<p<2$, $0<q \leq \infty$, $p^{\prime} = \tfrac{p}{p-1}$, and $f \in L_{p,q}[0,1]$, then
			\[
			\left(\sum_{k=1}^{\infty} k^{\tfrac{q}{p^{\prime}}-1} (a_k^*)^q \right)^{\tfrac{1}{q}}
			\leq c \, \|f\|_{L_{p,q}[0,1]}.
			\]
			
			\item If $2<p<\infty$, $0<q \leq \infty$, and $\sum_{k=1}^{\infty} k^{\tfrac{q}{p^{\prime}}-1}(a_k^*)^q < \infty$, then there exists $f \in L_{p,q}$ such that $a_n(f) = a_n$, and
			\[
			\|f\|_{L_{p,q}[0,1]} \leq c \left(\sum_{k=1}^{\infty} k^{\tfrac{q}{p^{\prime}}-1}(a_k^*)^q \right)^{\tfrac{1}{q}}.
			\]
		\end{enumerate}
		
		Here, for $p=q$ one recovers Paley’s inequality.  
		In the case $1<p<2$, all these inequalities provide necessary conditions for a function to belong to $L_p$ or $L_{p,q}$, while for $2<p<\infty$ they provide sufficient conditions.
		
		The problem of estimating Fourier coefficients of functions in Lorentz spaces $L_{2,q}$, with $2<q \leq \infty$, remained open until the late 20th century.  
		
	Only in 1997, Bochkarev \cite{16} established that, unlike in the case of $L_{p,r}$ with $1<p<2$, $1 \leq r \leq \infty$, in Lorentz spaces $L_{2,q}$, $2<q \leq \infty$, a direct analogue of the Hausdorff--Young--Riesz theorem does not hold. He obtained upper estimates for the Fourier coefficients of functions in $L_{2,r}$ that replace the Hausdorff--Young--Riesz theorem, and proved that for a certain class of multiplicative systems these estimates are sharp.	
	\begin{theorem}\cite{16}
		Let $\{\varphi_n\}_{n=1}^{\infty}$ be an orthonormal system of complex-valued functions on $[0,1]$ satisfying
		\[
		\|\varphi_n\|_{\infty} \leq M, \quad n=1,2,\ldots,
		\]
		and let $f \in L_{2,q}$ with $2<q \leq \infty$. Then the following inequality holds:
		\[
		\sup_{n \in \mathbb{N}}
		\frac{1}{n^{1/2} \, (\log(n+1))^{1/2 - 1/q}}
		\sum_{m=1}^n a_m^*
		\leq C \, \|f\|_{L_{2,q}},
		\]
		where $a_n$ are the Fourier coefficients with respect to the system $\{\varphi_n\}_{n=1}^{\infty}$.
		\end{theorem}	
	Further investigations of the behavior of Fourier coefficients of functions belonging to the space $L_{2,q}$ were carried out in the works \cite{MT13, MT15, N2000Izv}. In papers~\cite{22,23}, refinements of Bochkarev’s inequality were obtained independently of each other.\\ 
			Let $\Psi = \{\psi_{n_1}(x_1), \phi_{n_2}(x_2)\}_{n_1, n_2 \in \mathbb{N}}, \,\; (x_1,x_2) \in [0,1]^2$, be an orthonormal systems in $L_2[0,1]^2$ such that for every $n \in \mathbb{N}$ we have $|\psi_n(x_1)|\leq M_1$ and $|\psi_n(x_2)| \leq M_2$ almost everywhere.
			Let $f \in L_1{[0,1]^2}$ ,
		\[
		a_{k_1 k_2} = \int_0^1  \int_0^1 f(x_1, x_2) \overline{\psi_{k_1}(x_1)}\; \overline{\phi_{k_2}(x_2)}\, dx_1 dx_2,\;\;\;\;\;\;\; (k_1,k_2)\in \Bbb N^2
		\]
		the Fourier coefficients of a function with respect to the system~$\Psi$
		
			The decreasing rearrangement of $f$ is defined by
		$$
		f^{*}(t)=\inf \{\sigma: \mu\{x \in \Omega:|f(x)|>\sigma\} \leqslant t\}.
		$$

		$f^{*_ 1, *_2}\left(t_{1}, t_{2}\right)$ is obtained from $f(x)=f\left(x_{1}, x_{2}\right)$ by applying the nonincreasing rearrangement successively with respect to each variable $x_{1}$ and $x_{2}$, while keeping the other variable fixed. 
		where rearrangement is successively preformed with respect to each of the variables $x_{1}$ and $x_{2}$, whereas the other variable is assumed to be fixed (see \cite{N2000Izv}).
	Let $0<\bar p=(p_1,p_2)<\infty$ and $0<\bar q=(q_1,q_2)\leq\infty$.  
The anisotropic Lorentz space $L_{\bar p,\bar q}[0,1]^2$ is defined as the set of all measurable functions $f$ for which the following quantity is finite: 
		$$
		\|f\|_{L_{\bar{p}, \bar{q}}[0,1]^2}=	\left(\int_{0}^{1}\left(\int_{0}^{1}\left(t_{1}^{\frac{1}{p_{1}}} t_{2}^{\frac{1}{p_{2}}} f^{*_1, *_2}\left(t_{1}, t_{2}\right)\right)^{q_{1}} \frac{d t_{1}}{t_{1}}\right)^{\frac{q_{2}}{q_{1}}} \frac{d t_{2}}{t_{2}}\right)^{\frac{1}{q_{2}}}
		$$
	Here and throughout, expressions of the form  
$\left( \int |\phi(x)|^q\frac{dt}t\right) ^{1/q}, \; \left( \sum_k |a_k|^q\frac1k\right) ^{1/q}$ 
are understood, in the case $q=\infty$, as
$\sup_x |\phi(x)|, \; \sup_k|a_k|$
respectively.\\
Nursultanov~E.\,D. obtained Hardy–Littlewood–Stein \cite{Nursultanov2026} type inequalities for anisotropic Lorentz spaces.  
In particular, the following theorems hold:
\begin{theorem} \cite{Nursultanov2026}
Let  $0<\bar q=(q_1,q_2)\leq\infty $,   $\Psi = \{\psi_{n_1}(x_1), \phi_{n_2}(x_2)\}_{n_1, n_2 \in \mathbb{N}}$ be an orthonormal system of complex-valued functions uniformly bounded in the sense of complete uniformity.
	If  $1<\bar p=(p_1,p_2)<2$ and $f\in L_{\bar p,\bar q}[0,1]^2$,  then  $a(f)\in l^*_{\bar p',\bar q}$ and the inequality
	$$\left(\sum_{k_2 \in \mathbb{N}}\left(\sum_{k_1 \in \mathbb{N}}\left( k_2^{\frac{1}{p'_2}}k_1^{\frac1{p'_1}}a_{k_1 k_2}^{*_2, *_1}\right)^{q_1}\frac1{k_1} \right)^{\frac{q_2}{q_1}}\frac1{k_2}\right)^{\frac{1}{q_2}}\lesssim\|f\|_{L_{\bar p,\bar q}[0,1]^2},
		$$ holds.
If  $2<\bar p=(p_1,p_2)<\infty$ и $a(f)\in l^*_{\bar p,\bar q}$,  then  $f\in L_{\bar p',\bar q}[0,1]^2$ and the inequality
			$$
			\|f\|_{L_{\bar p',\bar q}[0,1]^2} \lesssim
		\left(\sum_{k_2 \in \mathbb{N}}\left(\sum_{k_1 \in \mathbb{N}}\left( k_2^{\frac{1}{p_2}}k_1^{\frac1{p_1}}a_{k_1 k_2}^{*_2, *_1}\right)^{q_1}\frac1{k_1} \right)^{\frac{q_2}{q_1}}\frac1{k_2}\right)^{\frac{1}{q_2}},
	$$ holds.
is valid, where $a(f)=\{a_{k_1 k_2}\}$ denotes the system of Fourier coefficients of $f$ with respect to~$\Psi$.
\end{theorem}
In the present paper, we investigate the summability properties of Fourier coefficients of functions belonging to the anisotropic Lorentz space.  
A key tool in our analysis is provided by the grand Lorentz spaces introduced in \cite{NursultanovRafeiroSuragan2025}.\\

Throughout this paper, the notation \( A \lesssim B \) means that $A \le c\, B ,$ for some constant \( c>0 \) independent of the essential parameters of the ambient space. The notation \( A \asymp B \) signifies a two--sided estimate, that is, both \( A \lesssim B \) and \( B \lesssim A \) hold. In the definitions below, we use the symbols
$\bar{p} = (p_1, p_2), \qquad \bar{q} = (q_1, q_2) \in \mathbb{R}_{+}^{2},
$ with all parameters considered in this generality. If 
$0 < p_1, p_2 \le \infty, \qquad \frac{1}{p_1} + \frac{1}{p_1'} = 1, \qquad
\frac{1}{p_2} + \frac{1}{p_2'} = 1,$ then we write \( \bar{0} < \bar{p} \le \bar{\infty} \) and $ \frac{1}{\bar{p}} + \frac{1}{\bar{p}'} = \bar{1}.
$ Moreover, if \( p_1 \le q_1 \) and \( p_2 \le q_2 \), we write \( \bar{p} \le \bar{q} \).
\section{Anisotropic Grand Lorentz Function Spaces}
\begin{definition} \cite{Manarbek}
Let $\bar{\theta} = (\theta_1, \theta_2)$, $-\infty < \bar{\theta} < \infty$, 
$\bar{p} = (p_1, p_2)$, $0 < \bar{p} \le \infty$, 
$\bar{q} = (q_1, q_2)$, $0 < \bar{q} \le \infty$.
The \emph{anisotropic grand Lorentz space} 
$G^{\bar{\theta}}L_{\bar{p},\bar{q}}[0,1]^2$ 
is defined as the set of all measurable functions $f$ for which the following quasi-norms are finite.  \\
For $0 < \bar q < \infty$:
\[
\|f\|_{G^{\bar{\theta}}L_{\bar{p},\bar{q}}[0,1]^2} =
\begin{cases}
\displaystyle
\sup_{0 < \varepsilon_1, \varepsilon_2 \le 1}
\varepsilon_1^{\theta_1} \varepsilon_2^{\theta_2}
\left(
  \int_0^1
  \left(
    \int_0^1
    \left(
      t_1^{\frac{1}{p_1} + \varepsilon_1}
      t_2^{\frac{1}{p_2} + \varepsilon_2}
      f^{*1,*2}(t_1,t_2)
    \right)^{q_1}
    \frac{dt_1}{t_1}
  \right)^{\frac{q_2}{q_1}}
  \frac{dt_2}{t_2}
\right)^{\frac{1}{q_2}}, \\
\text{for } \theta_1 \ge 0, \ \theta_2 \ge 0, 0<p_i\leq\infty, i=1,2, \\[2ex]
\displaystyle
\inf_{\substack{
0 < \varepsilon_1 \le \frac{1}{p_1} \\
0 < \varepsilon_2 \le \frac{1}{p_2}}}
\varepsilon_1^{\theta_1} \varepsilon_2^{\theta_2}
\left(
  \int_0^1
  \left(
    \int_0^1
    \left(
      t_1^{\frac{1}{p_1} - \varepsilon_1}
      t_2^{\frac{1}{p_2} - \varepsilon_2}
      f^{*1,*2}(t_1,t_2)
    \right)^{q_1}
    \frac{dt_1}{t_1}
  \right)^{\frac{q_2}{q_1}}
  \frac{dt_2}{t_2}
\right)^{\frac{1}{q_2}},\\
\text{for } \theta_1 < 0, \ \theta_2 < 0, 0<p_i\leq\infty, i=1,2.
\end{cases}\]\\
\begin{remark}
 Unlike the classical anisotropic Lorentz spaces  $L_{\bar{p}, \bar{q}}[0,1]^{2}$ the grand anisotropic Lorentz spaces $G^{\bar{\theta}}L_{\bar{p},\bar{q}}[0,1]^2$  have meaning in cases $p_i=\infty, \theta_i>0$.\\
\end{remark}
\begin{remark}
These spaces were introduced and their properties studied in the works \cite{NursultanovRafeiroSuragan2025}--\cite{Manarbek}.  In case when $\bar\theta=(0,0)$, we have ${G^{\bar{\theta}}L_{\bar{p},\bar{q}}[0,1]^2}=L_{\bar p,\bar q}[0,1]^2$.  Hence, the anisotropic grand Lorentz space includes in it scales the anisotropic Lorentz space, and the parameter $\bar\theta=(\theta_ 1, \theta_2)$ makes the associated scale of grand spaces more flexible.
\end{remark}
 The grand anisotropic Lorentz spaces allow us to consider the scale of spaces ${G^{\bar{\theta}}L_{\bar{p},\bar{q}}[0,1]^2}$. For $\bar{\theta}>\overline{0}$ the following property holds:
\begin{equation}
		{G^{-\bar{\theta}}L_{\bar{p},\bar{q}}[0,1]^2}. \hookrightarrow L_{\bar{p}, \bar{q}}[0,1]^2 \hookrightarrow {G^{\bar{\theta}}L_{\bar{p},\bar{q}}[0,1]^2}. 
\end{equation}
\begin{lemma}\label{L1} \cite{Manarbek}
	Let  $0<p_{1}, p_{2}<\infty$. Then for $\theta_{1}, \theta_{2}>0$
		$$
		\|f\|_{G^{\bar{\theta}}L_{\bar{p},\bar{\infty}}[0,1]^2} \asymp \sup _{t_{1}, t_{2}>0} \frac{t_{1}^{\frac{1}{p_{1}}} \cdot t_{2}^{\frac{1}{p_{2}}}}{\left|\ln t_{1}\right|^{\theta_{1}}\left|\ln t_{2}\right|^{\theta_{2}}} f^{*_{1}, *_{2}}\left(t_{1}, t_{2}\right)
		$$
		$$
		\begin{aligned}
			\|f\|_{G^{\bar{\theta}}L_{\bar{p},\bar{q}}[0,1]^2} & \lesssim\left(\int_{0}^{1}\left(\int_{0}^{1}\left(\frac{t_{1}^{\frac{1}{p_{1}}} \cdot t_{2}^{\frac{1}{p_{2}}}}{\left|\ln t_{1}\right|^{\theta_{1}}\left|\ln t_{2}\right|^{\theta_{2}}} f^{*_{1}, *_{2}}\left(t_{1}, t_{2}\right)\right)^{q_{1}} \frac{d t_{1}}{t_{1}}\right)^{\frac{q_{2}}{q_{1}}} \frac{d t_{2}}{t_{2}}\right)^{q_{2}} \\
			\|f\|_{G^{-\bar{\theta}}L_{\bar{p},\bar{q}}[0,1]^2}& \gtrsim\left(\int_{0}^{1}\left(\int_{0}^{1}\left(t_{1}^{\frac{1}{p_{1}}}\left|\ln t_{1}\right|^{\theta_{1}} t_{2}^{\frac{1}{p_{2}}}\left|\ln t_{2}\right|^{\theta_{2}} f^{*_ 1, *_ 2}\left(t_{1}, t_{2}\right)\right)^{q_{1}} \frac{d t_{1}}{t_{1}}\right)^{\frac{q_{2}}{q_{1}}} \frac{d t_{2}}{t_{2}}\right)^{q_{2}}
		\end{aligned}
		$$
	\end{lemma}
From the definition of $G^{\bar{\theta}}L_{\bar{p},\bar{q}}[0,1]^2$ and well-known results on Lorentz spaces it is easy to prove the following properties \cite{Manarbek}:\\
	(P.1) if $\bar{\theta} \leqslant \bar{s}$, then $G^{\bar{\theta}}L_{\bar{p},\bar{q}}[0,1]^2\hookrightarrow G^{\bar{s}}L_{\bar{p},\bar{q}}[0,1]^2$,\\
	(P.2) if $\bar{q}<\bar{r}$, then $G^{\bar{\theta}}L_{\bar{p},\bar{q}}[0,1]^2\hookrightarrow G^{\bar{\theta}}L_{\bar{p},\bar{r}}[0,1]^2$,\\
    (P.3) if $\bar p<\bar p_1, 0<\bar q,\bar q_1 \leq\infty$, then $G^{\bar{\theta}}L_{\bar{p_1},\bar{q_1}}[0,1]^2\hookrightarrow G^{\bar{\theta}}L_{\bar{p},\bar{q}}[0,1]^2$,\\
	(P.4) if $\overline{0}<\bar{\delta}<\overline{1}$ and $\bar{\theta}>\overline{0}$, then
	$$
	\|f\|_{G^{\bar{\theta}}L_{\bar{p},\bar{q}}[0,1]^2} \asymp \sup _{\substack{0<\varepsilon_{1} \leqslant \delta_{1} \\ 0<\varepsilon_{2} \leqslant \delta_{2}}} \varepsilon_{1}^{\theta_{1}} \varepsilon_{2}^{\theta_{2}}\left(\int_{0}^{1}\left(\int_{0}^{1}\left(t_{1}^{\frac{1}{p_{1}}+\varepsilon_{1}} t_{2}^{\frac{1}{p_{2}}+\varepsilon_{2}} f^{*_ 1, *_ 2}\left(t_{1}, t_{2}\right)\right)^{q_{1}} \frac{d t_{1}}{t_{1}}\right)^{\frac{q_{2}}{q_{1}}} \frac{d t_{2}}{t_{2}}\right)^{q_{2}}
	$$
	
	and, if $0<\bar{\delta}<1 / \bar{p}$,
	
	$$
	\|f\|_{G^{-\bar{\theta}}L_{\bar{p},\bar{q}}[0,1]^2} \asymp \inf _{\substack{0<\varepsilon_{1}<\delta_{1} \\ 0<\varepsilon_{2}<\delta_{2}}} \varepsilon_{1}^{-\theta_{1}} \varepsilon_{2}^{-\theta_{2}}\left(\int_{0}^{1}\left(\int_{0}^{1}\left(t_{1}^{\frac{1}{p_{1}}-\varepsilon_{1}} t_{2}^{\frac{1}{p_{2}}-\varepsilon_{2}} f^{*_{1}, *_{2}}\left(t_{1}, t_{2}\right)\right)^{q_{1}} \frac{d t_{1}}{t_{1}}\right)^{\frac{q_{2}}{q_{1}}} \frac{d t_{2}}{t_{2}}\right)^{q_{2}}
	$$
	(P.5) Let  $0<\bar p\leq\infty, \bar\theta<\bar\lambda$ and $\bar\theta-\frac{1}{\bar q}$ then the following embedding holds: $G^{\bar{\theta}} L_{{\bar{p}}, \bar{\tau}}[0,1]^2 \hookrightarrow G^{\bar{\lambda}}L_{{\bar{p}},\bar{\tau}}[0,1]^2$,\\
	(P.6) if $0<\bar{p}<\infty, 0<\bar{\tau} \leqslant \infty, \bar{\theta}>0$, then
	$$
	\begin{aligned}
		\|f\|_{G^{\bar{\theta}} L {\bar{p}}, \bar{\tau}} & =\sup _{1<k_{1}, k_{2}} k_{1}^{-\theta_{1}} k_{2}^{-\theta_{2}} \\
		& \times\left(\sum_{m_{1}=-\infty}^{0}\left(\sum_{m_{2}=-\infty}^{0}\left(2^{m_{1}\left(\frac{1}{p_{1}}+\frac{1}{k_{1}}\right)} 2^{m_{2}\left(\frac{1}{p_{2}}+\frac{1}{k_{2}}\right)} f^{*_1,*_2}\left(2^{m_{1}}, 2^{m_{2}}\right)\right)^{\tau_{1}}\right)^{\frac{\tau_{2}}{\tau_{1}}}\right)^{\frac{1}{\tau_{2}}}
	\end{aligned}
	$$
\end{definition}	
We will need a certain discrete version of the grand Lorentz space $G^{\bar{\theta}}l^*_{\bar{p},\bar{q}}[0,1]^2$.\\
Let $0\le \bar\theta=(\theta_1,\theta_2)$, $0<\bar q=(q_1,q_2)\le\infty$, and  
$2\le \bar p=(p_1,p_2)\le\infty$.  
By $G^{\bar{\theta}}l^*_{\bar{p},\bar{q}}[0,1]$ we denote the space of all sequences of complex numbers  
$a=\{a_m\}_{m\in\mathbb{Z}^2}$ for which the quantity
\begin{equation}\label{e8}
\|a\|_{G^{\bar\theta}\ell^{*}_{\bar p,\bar q}}
=\sup_{0 < \varepsilon_1, \varepsilon_2 \le 1}
\varepsilon_1^{\theta_1} \varepsilon_2^{\theta_2}
\left(
\sum_{k_2=0}^{\infty}
\left(
\sum_{k_1=0}^{\infty}
\left(
2^{k_1(\frac{1}{p_1}+\varepsilon_1)+k_2(\frac{1}{p_2}+\varepsilon_2)}
\left[
\frac{1}{2^{k_1+k_2}}
\sum_{m_2=1}^{2^{k_2}}
\sum_{m_1=1}^{2^{k_1}}
\left(a^{*_{2},*_{1}}_{m_1 m_2}\right)^2
\right]^{1/2}
\right)^{q_1}
\right)^{\frac{q_2}{q_1}}
\right)^{\frac{1}{q_2}}
\end{equation}
is finite.

\begin{remark}
Note that the iterated nonincreasing rearrangement of the sequence  
$a=\{a_m\}_{m\in\mathbb{Z}^2}$ is taken in the reverse order with respect to the order of summation in~\eqref{e8}.
\end{remark}

\section{Some inequalities} 
\begin{lemma}[Karamata-type Inequality]\cite[p. 10]{bookNursultanov} \label{karamato} 
Let  $f(x), g(x)$ be nonnegative, nonincreasing functions
	n $(0,\infty)$, and suppose that for every $t>0$,
		$$
		\int\limits_{0}^t f(x)dx\geq \int\limits_{0}^t g(x)dx,
		$$
		$$
		\int\limits_{0}^\infty f(x)dx= \int\limits_{0}^\infty g(x)dx,
		$$
and moreover,\\
\textbf{(a)} if $p\ge 1$, we have
$$
\int_{0}^{\infty} |f(x)|^{p}\,dx \;\ge\; \int_{0}^{\infty} |g(x)|^{p}\,dx;
$$

\textbf{(b)} if $p\le 1$, we have
$$
\int_{0}^{\infty} |f(x)|^{p}\,dx \;\le\; \int_{0}^{\infty} |g(x)|^{p}\,dx.
$$
\end{lemma}	
			
\begin{lemma} \label{mink}
		
	Let $0< p\le q\le\infty$, then:
		
		{\bf a)}
		$$
		\left ( \int_0^1
		\left(\int_0^1  | f(x_1,x_2)|^q dx_2
		\right)^{p/q}
		dx_1 \right)^{1/p} \geq
		\left ( \int_0^1
		\left (\int_0^1  (f^{*_1}(t_1,x_2))^q dx_2
		\right )^{p/q}
		dx_1   \right )^{1/p}.
		$$
		{\bf b)}
		$$
		\left ( \int_0^1
		\left (\int_0^1  | f(x_1,x_2)|^p dx_2
		\right )^{q/p}
		dx_1   \right )^{1/q} \leq
		\left ( \int_0^1
		\left (\int_0^1  (f^{*_1}(t_1,x_2))^p dx_2
		\right )^{q/p}
		dx_1     \right )^{1/q}.
		$$
	\end{lemma}
	\begin{proof}a) Let $\Phi^*(t_2) $ denote the decreasing rearrangement of the function $\Phi(x_2)=\int_0^1 |f(x_1,x_2)|^q dx_1$. If $G(t_2)= \int_0^1 (f^{*_2}(x_1,t_2))^qd x_1
		,$ then
		$$
		\Phi^*(t_2)\leq G(t_2), \;\;\;\; t\in (0,1)
		$$
		$$
		\Phi^*(1)=G(1).
		$$
Applying Lemma~\ref{karamato} to these functions and noting that $ \alpha = \frac{p}{q} \le 1,$
we obtain the desired inequality in part (a). The argument for part (b) is analogous.
\end{proof}
\begin{lemma}
[Hardy's inequality] \label{hardy} \cite[p. 14]{bookNursultanov}
		Let  $\alpha>0, \quad \beta=\max \left(\frac{1}{q}, \frac{1}{r}\right), 0<q \leq \infty $
		then the following inequality holds:\\
		\[
		\begin{aligned}
			\left(\int_0^{\infty}\left(t^{-\alpha}\left(\int_0^t\left(f^*(s)\right)^r d s\right)^{1 /r}\right)^q \frac{d t}{t}\right)^{1 / q} \leq C \frac{1}{\alpha^\beta}\left(\int_0^{\infty}\left(t^{\frac{1}{r}-\alpha} f^*(t)\right)^q \frac{d t}{t}\right)^{1 / q}, \\
			\left(\int_0^{\infty}\left(t^\alpha\left(\int_t^{\infty}\left(f^*(s)\right)^r d s\right)^{1 / r}\right)^q \frac{d t}{t}\right)^{1/q} \leq C\frac{1}{\alpha^\beta}\left(\int_0^{\infty}\left(t^{\alpha+\frac{1}{r}} f^*(t)\right)^q \frac{d t}{t}\right)^{1 / q},
		\end{aligned}
		\]
		here the corresponding constant do not depend on the parameter~$\alpha$.
	\end{lemma}
	\begin{proof}
		
		Let us make a change of variables in the inner integral. We will prove Hardy’s inequality in the form that we need:
		$$
		\left(\int_0^{\infty}\left(t^{-\alpha}\left(\int_0^t\left(f^*(s)\right)^r d s\right)^{1 / r}\right)^q \frac{d t}{t}\right)^{1 / q}=\left( \int_0^\infty\left(t^{-\alpha}\left(\int_0^1\left(f^*(s t)\right)^r t d s\right)^{1 / r}\right)^q \frac{dt}{t}\right) ^{1/q}
		$$
		Let $r\leq q$, then applying the generalized Minkowski inequality and performing the inverse change of variables, we obtain:
		\[
		\begin{gathered}
			\leq\left(\int_0^{1}\left(\int_0^\infty\left(f^*(st)t^{\frac{1}{r}-\alpha}\right)^q \frac{d t}{t}\right)^\frac{r}{q} ds\right)^{1 / r}= \\
			=\left(\int_0^1\left(s^{\alpha-\frac{1}{r}}\right)^r d s\right)^{1 / r}\left(\int_0^{\infty}\left(t^{\frac{1}{r}-\alpha} f^*(t)\right)^q \frac{d t}{t}\right)^{1 / q}= \\ =\frac{1}{(\alpha r)^{1 / r}}\left(\int_0^{\infty}\left(t^{\frac{1}{\tau}-\alpha} f^*(t)\right)^q \frac{d t}{t}\right)^{1 / q}
		\end{gathered}
		\]
		Similarly, \\
		$\begin{aligned} & \left(\int_0^{\infty}\left(t^\alpha\left(\int_t^{\infty}\left(f^*(s)\right)^r d s\right)^{1 / r}\right)^q \frac{d t}{t}\right)^{1 /q}=\left(\int_0^{\infty}\left(t^\alpha\left(\int_1^{\infty}\left(f^*(s t)\right)^r t d s\right)^{1 /r}\right)^q \frac{d t}{t}\right)^{1/q} \\ \leq & \left(\int_1^{\infty}\left(\int_0^{\infty}\left(t^{\alpha+\frac{1}{r}} f^*(s t)\right)^q \frac{d t}{t}\right)^{\frac{r}{q}} d s\right)^{1 / r}= \\ = & \left(\int_1^{\infty}\left(s^{-\alpha-\frac{1}{r}}\right)^r d s\right)^{1 / r}\left(\int_0^{\infty}\left(t^{\alpha+\frac{1}{r}} f(t)\right)^t \frac{d t}{t}\right)^{1 / q}= \\ = & \frac{1}{(\alpha r)^{1 / r}}\left(\int_0^{\infty}\left(t^{\alpha+\frac{1}{r}} f^*(t)\right)^q \frac{d t}{t}\right)^{1 / q}.\end{aligned}$\\
        \vspace{0.5cm}
		Let $q<r$. Since $f^*$ is decreasing, we obtain: \\
		$\begin{gathered}\left(\int_0^{\infty}\left(t^{-\alpha} \int_0^t\left(f^*(s)\right)^{r} d s\right)^q \frac{d t}{t}\right)^{1 / q} \asymp \left(\sum_{k \in \mathbb{Z}}\left(2^{-\alpha k}\left(\sum_{m=-\infty}^k 2^m\left(f^*\left(2^m\right)\right)^r\right)^{1 / r}\right)^q\right)^{1 / q}.\end{gathered}$\\
		We apply Jensen's inequality:\\
		$$	
		\leq\left(\sum_{k \in \mathbb{Z}} 2^{-\alpha qk } \sum_{m=-\infty}^k\left(2^{\frac mr} f^*\left(2^m\right)\right)^q\right)^{1 /q} =\left(\sum_{k \in \mathbb{Z}}\left(2^{\frac{m}{r}} f^*\left(2^m\right)\right)^q \sum_{k=m}^{\infty} 2^{-\alpha q k}\right)^{1 / q}
		$$
		$$
		\asymp \frac{1}{\alpha^\frac{1}{q}}\left(\sum_{m \in \mathbb{Z}}\left(2^{\left(-\alpha+\frac{1}{\tau}\right) m} f^*\left(2^m\right)\right)^q\right)^{1 / q}  \asymp \frac{1}{\alpha^{1 / q}}\left(\int_0^{\infty}\left(t^{-\alpha+\frac1r} f(t)\right)^q\frac{d t}{t}\right)^{1 / q}. 
		$$
		Notice that the corresponding constant does not depend on the parameter~$\alpha$. Similarly, 
		$$
	\left(\int_0^{\infty}\left(t^\alpha\left({\int_t^\infty}\left(f^*(s)\right)^r d s\right)^{1 / r}\right)^q \frac{d t}{t}\right)^{1 / q} \asymp \\
		\left(\sum_{k \in \mathbb{Z}}\left(2^{\alpha k}\left(\sum_{m=k}^{\infty} 2^m\left(f^*\left(2^m\right)\right)^r\right)^{1 / r}\right)^q\right)^{1 / q}
		$$
		$$
		\leq \left(\sum_{k \in \mathbb{Z}}2^{\alpha q k}\sum_{m=k}^{\infty} 2^{\frac{qm}r}\left(f^*\left(2^m\right)\right)^q\right)^{1 / q}
		=\left(\sum_{m \in \Bbb Z}\left(2^{\frac{m}{r}} f^*\left(2^m\right)\right)^q \sum_{k=-\infty}^m 2^{\alpha k q}\right)^{1 / q} 
		$$
		$$
		\asymp\frac{1}{\alpha^{1 / q}}\left(\sum_{m \in\Bbb Z}^{\infty}\left(2^{\left(\alpha+\frac{1}{r}\right)m} f^*\left(2^m\right)\right)^q\right)^{1 / q}
		\asymp  \frac{1}{\alpha^\frac{1}{q}}\left(\int_0^{\infty}\left(t^{\alpha+\frac{1}{r}} f^*\left(t\right)\right)^q \frac{d t}{t}\right)^{1 / q} \\
		$$
	\end{proof}

	\section{Anisotropic Interpolation Method} 
Let $ A_1 $ be a Banach space and $ A_2 $ be a functional Banach lattice. Let us denote by $ {\bf A} = (A_1, A_2) $ the space of $
	A_1 $ -- valued measurable functions such that $ \| f\|_{A_1} \in A_2 $ with the norm $ \| f \| = \| \| f(x)\|_{A_1} \|_{A_2} $.
	
	The space $ {\bf A}=(A_1, \ldots, A_n)$ is defined inductively. We will call it a {\it space with a mixed metric} \index{Space with a mixed metric}.
	
	Let ${ \bf A_0} = (A_1^{0},...,A_n^{0}),\; {\bf A_1} =
	(A_1^1,...,A_n^1) $ are two spaces with mixed metric, $ E =
	\{\varepsilon=(\varepsilon_1,...,\varepsilon_n):\varepsilon_i=0,$ or $
	\varepsilon_i=1,\;\;i=1,...,n \}$ be the set of all vertices of the unit cube $[0,1]^n$ in $\Bbb R^n$. For an arbitrary $\varepsilon\in E$, we define the space 
	${ {\bf A}_\varepsilon} =
	(A_1^{\varepsilon_1},...,A_n^{\varepsilon_n})$ with the norm
	$$
	\|a\|_{\bf A} = \|...\|a\|_{A_1^{\varepsilon_1}}...\|_{A_n^{\varepsilon_n}}.
	$$
	
	A pair of spaces with mixed metrics
	$ {\bf A_0} = (A_1^{0},...,A_n^{0})$ and ${ \bf A_1} = (A_1^1,...,A_n^1) $
	is called {\it compatible}\index{Compatible spaces with mixed metrics}, if there exists a linear topological Hausdorff space that contains the spaces
	${\bf A_\varepsilon}$ for all ${\varepsilon\in
		E}$ as subspaces.
	
	Let us define the functional: \index{Functional Peetre $K(t,a;{\bf A_0},{\bf A_1})$}
	
	\begin{equation}\label{2.9}
		K(t,a;{\bf A_0},{\bf A_1}) = \inf\{\sum_{\varepsilon\in E} t^{\varepsilon}
		\|a_{\varepsilon}\|_{\bf {A_{\varepsilon}}}\;:\;a=\sum_{\varepsilon\in E}
		a_\varepsilon,\;a_\varepsilon\in {\bf A_\varepsilon}\},
	\end{equation}
	where $ t^{\varepsilon} = t_1^{\varepsilon_1}\ldots t_n^{\varepsilon_n}.$

	Let $ 0 < \bar\theta=(\theta_1,\ldots,\theta_n) < 1,\;\; 0 <
	{\bar q} =(q_1,\ldots,q_n)\leq \infty.$
	By ${ \bf A}_{\bar\theta, \bar q} = ({\bf A_0, A_1})_{\theta,{\bar q}} $ \index{Space ${ \bf A}_{\bar\theta, \bar q} = ({\bf A_0, A_1})_{\theta,{\bf q^\star}} $ }
	we denote the linear subset $ \sum_{\varepsilon\in E}{\bf
		A}_\varepsilon,$ such that for whose elements the following is true:
	
	\begin{equation}\label{2.10}
		\|a\|_{\bf A_{\bar \theta, \bar q}} =
	\end{equation}
	$$
	=\left(\int_0^\infty\ldots\left(\int_0^\infty
	\left(t_1^{-\theta_1}\ldots
	t_n^{-\theta_n}K(t_1,\ldots,t_n;a)\right)^{q_{1}}
	\frac{dt_{1}}{t_{1}}\right)^\frac{q_{2}}{q_{1}]}\ldots
	\frac{dt_{n}}{t_{n}}\right)^\frac1{q_{n}}.
	$$
	
	\begin{lemma}\label{Le0}  \cite{9} Let $ \bar \theta\in (0,1)^n$ and $0<{\bar q}\leq \infty$.
		
		(i) If $\bar q_1 \leq \bar q,$ then
		$$
		({\bf A_0,A_1})_{\bar \theta, \bar q_1} \hookrightarrow ({\bf
			A_0,A_1})_{\bar\theta, \bar q}.
		$$
		
		(ii)  If  $T$
		is a quasi-linear operator such that for all  $\varepsilon\in E$
		$$
		T : \bf {A_{\varepsilon}} \to \bf {B_\varepsilon},
		$$
		with norm $ M_\varepsilon$, then
		$$
		T :{ \bf A}_{\bar \theta, {\bar q}} \to {\bf B}_{\bar \theta, {\bar  q}},\;\;
		$$
		with $\|T\|\le\max_{\varepsilon\in E}M_\varepsilon$.
		(iii)\label{iii}
		$$
		({\bf A_0,A_1})_{\bar\theta, \bar q}= ({\bf A_1,A_0})_{\overline {1-\theta},\; \bar q}.
		$$
	\end{lemma}
		We will use an anisotropic interpolation method in case $n=2$.\\
		Let $0<\bar p=(p_1,p_2)\leq \infty$.The Lebesgue space with mixed norm will be denoted by $L_{\overline{p}}(\mathbb{T}^2)$ or $L_{(p_1,p_2)}(\mathbb{T}^2)$, where norm is the following: 
	\begin{gather}\label{4.3}
		\|f\|_{L_{\overline{p}}(\mathbb{T}^2)}=\left(\int_0^{1}\left(\int_0^{1}\left(f\left(x, x_2\right)\right)^{p_1} d x_1\right)^\frac{p_2}{p_1} d x_2\right)^\frac{1}{p_2}. 
	\end{gather}
	Notice that $L_{\overline{p}}$ space is identical to the anisotropic Lorentz space $L_{\overline{p},\overline{q}}$ only when $p_1=p_2=q_1=q_2$. 

	\begin{theorem}\label{Te1} Let $0<\bar q\leq\infty, \bar \theta\in(0,1)^2$ and $\beta_i=\max\{\frac12,\frac1{q_i}\}, \; i=1,2$. 
		\begin{equation}\label{int0}
		(L_{\bar 1}, L_{\bar 2})_{\bar \theta, \bar q}=L_{\bar p,\bar q}
	\end{equation}
		where $\frac 1{ p_i}=1-\frac{ \theta_i}2, \;i=1,2$.
Moreover,
	\begin{equation}\label{int}
	\|f\|_{(L_{\bar 1}, L_{\bar 2})_{\bar \theta, \bar q}}\leq 6 D(\theta_1,\theta_2)\|f\|_{L_{\bar p,\bar q}}
\end{equation}
	where
	\begin{equation}\label{const}
	D(\theta_1,\theta_2)=\max\left\lbrace \frac1{\theta_1\theta_2},\frac1{(1-\theta_1)^{\beta_1}\theta_2},\frac1{(1-\theta_2)^{\beta_2}\theta_1},
	\frac1{(1-\theta_2)^{\beta_2}(1-\theta_1)^{\beta_1}} \right\rbrace 
	\end{equation}
	\end{theorem}
	\begin{remark}
Identity \eqref{int0} shows that the interpolation of Lebesgue spaces with a mixed metric, with respect to the anisotropic interpolation method, is described by an anisotropic Lorentz space.  
Equality \eqref{int0} was previously established by Nursultanov~E.\,D.~\cite{N2000Izv}.  
The contribution of the present theorem is the determination of the exact dependence of the constant in inequality~\eqref{int} on the parameters $\theta_{1}$ and $\theta_{2}$.  
This estimate will play an essential role in the proofs of the main results of this paper.
	\end{remark}
	\begin{proof} Let $t_1>0, t_2>0$.
	Consider the decomposition $ f = f_{0 0} + f_{0 1} +
	f_{1 0} + f_{1 1},$ which depends on the parameters $t_{1}$ and $t_{2}$ and is constructed as follows.
	Let $w_1 = t_1^2,
	\;w_2 =t_2^2,\;\;
	\phi_{\Omega_{x_2}}(x_1) is $ the characteristic function of the set $\Omega_{x_2} = \{(x_1,x_2)\;:\;
	|f(x_1,x_2)| > f^{\ast_1}(w_1,x_2)\}\bigcup e_{x_2}, $ where $e_{x_2} - $ 
	is a measurable subset of $\{(x_1,x_2) : |f(x_1,x_2)| = f^{\ast_1}(w_1,x_2)\}$ 
	such that  $\mu_1(\Omega_{x_2}) = w_1.$
	Such a set can always be chosen, since for fixed $x_{2}$ we have
	$$
	\mu_1\{(x_1,x_2) : |f(x_1,x_2)| > f^{\ast_1}(w_1,x_2)\} \leq w_1 \leq
	$$
	$$
	\leq \mu_1\{(x_1,x_2) : |f(x_1,x_2)| \geq f^{\ast}(w_1,x_2)\}.
	$$
	Define:
	$$
	g_0(x_1,x_2) = f(x_1,x_2) \phi_{\Omega_{x_2}}(x_1)
	$$
	$$
	g_1(x_1,x_2) = f(x_1,x_2) - g_0(x_1,x_2) .
	$$
	$g_0$ and $g_1$ Each of the functions $g_{0}$ and $g_{1}$ will again be decomposed:
	$$
	g_0 = f_{0 0} + f_{0 1},\;\;\; g_1 = f_{1 0} + f_{1 1} .
	$$
Let  $W_0 = \{x_2\in {\mathbb T} : \| g_0(\cdot,x_2) \|_{L_1} >
	\left ( \|g_0(\cdot,x_2)\|_{L_1} \right )^{\ast_2}(w_2) \} \bigcup e_0,$
	где $ e_0 \subset \{ x_2\in {\mathbb T} : \| g_0(\cdot,x_2) 
	\|_{L_1} =\left ( \| g_0(\cdot,x_2) \|_{p_1^{0} 1} \right 
	)^{\ast_2}(w_2) \}$
	$$
	\mu_2(W_0) = w_2
	$$
	and
	$$
	W_1 = \{ x_2\in {\mathbb T} : \| g_1(.,x_2) \|_{L_2} > \left (
	\| g_1(.,x_2) \|_{L_2} \right )^{\ast_2}(w_2) \} \bigcup e_1 ,
	$$
	where $ e_1\subset \{ x_2\in {\mathbb T} : \| g_1(.,x_2) \|_{L_2} =
	\| g_1(.,x_2) \|_{p_1^1 1} \} $ и $ \mu_2(W_1) = w_2.$ Тогда:
	$$
	f_{0 0}(x_1,x_2) = g_0(x_1,x_2) \phi_{W_0}(x_2),
	f_{0 1}(x_1,x_2) = g_0(x_1,x_2) - f_{0 0}(x_1,x_2),
	$$
	$$
	f_{1 0}(x_1,x_2) = g_1(x_1,x_2) \phi_{W_1}(x_1),
	f_{0 1}(x_1,x_2) = g_0(x_1,x_2) - f_{0 0}(x_1,x_2) .
	$$
	We then set:
	$$
	f = f_{0 0} + f_{0 1} + f_{1 0} + f_{1 1}.
	$$
	Then
	$$
	K(t_1,t_2;f)\leq \|f_{00}\|_{L_{(1,1)}}+t_1 \|f_{10}\|_{L_{(2,1)}}+t_2\|f_{01}\|_{L_{(1,2)}}+t_1t_2 \|f_{11}\|_{L_{(2,2)}}
	$$
	where
	
	$$
	\|f_{00}\|_{L_{(1,1)}}=\int_0^{w_2}\left(  \int_0^{w_1} f^{*_1}(s_1,\cdot)ds_1\right)^{*_2} (s_2)ds_2,
	$$
	
		$$
	\|f_{10}\|_{L_{(2,1)}}=\int_0^{w_2}\left( \left(  \int_{w_1}^\infty (f^{*_1}(s_1,\cdot))^2ds_1\right)^{\frac12} \right)^{*_2} (s_2)ds_2,
	$$
	
	$$
	\|f_{01}\|_{L_{(1,2)}}=\left( \int_{w_2}^\infty\left( \left(  \int_0^{w_1} f^{*_1}(s_1,\cdot)ds_1 \right)^{*_2}(s_2)\right)^2 ds_2\right)^{\frac12} ,
	$$
	$$
	\|f_{11}\|_{L_{(2,2)}}=\left( \int_{w_2}^\infty\left(   \int_0^{w_1} (f^{*_1}(s_1,\cdot))^2ds_1 \right)^{*_2}(s_2)ds_2\right)^{\frac12},
	$$
	here we assume $f(x_1,x_2)=0$ when $(x_1, x_2)\notin [0,1)^2$.
Next, applying Lemma~\ref{mink} and the inequality
	$$
	g^*(t)\leq \frac1t\int_0^tg^*(s)ds=\sup_{|e|=t}\frac1{|e|} \int_e|g(x)|dx
	$$
we get  
	 $$
	 \|f_{00}\|_{L_{(1,1)}}\leq\int_0^{w_2} \sup_{|e|=s_2}\frac1{|e|} \int_e \int_0^{w_1} f^{*_1}(s_1,x_2)ds_1dx_2ds_2,
	 $$
	 $$
	 \leq\int_0^{w_2} \int_0^{w_1}\sup_{|e|=s_2}\frac1{|e|} \int_e  f^{*_1}(s_1,x_2)dx_2ds_1ds_2
	 $$
	 $$
	 =\int_0^{w_2} \int_0^{w_1}\frac1{s_2} \int_0^{s_2}  f^{*_1*_2}(s_1,r_2)dr_2ds_1ds_2,
	 $$
	 $$
	 \|f_{10}\|_{L_{(2,1)}}\leq\int_0^{w_2} \left(  \int_{w_1}^\infty (f^{*_1*_2}(s_1,s_2))^2ds_1\right)^{\frac12} ds_2,
	 $$
	 
	 $$
	 \|f_{01}\|_{L_{(1,2)}}\leq\left( \int_{w_2}^\infty\left( \sup_{|e|=s_2}\frac1{|e|} \int_e   \int_0^{w_1} f^{*_1}(s_1,x_2)ds_1 dx_2\right)^2 ds_2\right)^{\frac12}
	 $$
	 $$
	 \leq\left( \int_{w_2}^\infty\left(  \int_0^{w_1} \frac1{s_2} \int_0^{s_2}  f^{*_1*_2}(s_1,r_2)dr_2ds_1 \right)^2 ds_2\right)^{\frac12} ,
	 $$
	 $$
	 \|f_{11}\|_{L_{(2,2)}}\leq\left( \int_{w_2}^\infty   \int_0^{w_1} (f^{*_1*_2}(s_1,s_2))^2ds_1 ds_2\right)^{\frac12}.
	 $$	
Therefore we have
	$$
	\|f\|_{(L_{\bar (1)}, L_{\bar (2)})_{\bar\theta,\bar q}}=\left( \int_0^\infty\left( \int_0^\infty\left( t_1^{-\theta_1}t_2^{-\theta_2}	K(t_1,t_2;f;L_{(1,1)},L_{(2,2)})\right)^{q_1} \frac{dt_1}{t_1}\right) ^{\frac{q_2}{q_1}}\frac{dt_2}{t_2}\right) ^{\frac1{q_2}}
	$$
	$$
	\leq I_{00}+I_{10}+I_{01}+I_{11}.
	$$
Using Hardy’s inequality (see Lemma~\ref{hardy}), we estimate each term:
	$$
	I_{00}=\left( \int_0^\infty\left( \int_0^\infty\left( t_1^{-\theta_1}t_2^{-\theta_2}	\int_0^{w_2} \int_0^{w_1}\frac1{s_2} \int_0^{s_2}  f^{*_1*_2}(s_1,r_2)dr_2ds_1ds_2\right)^{q_1} \frac{dt_1}{t_1}\right) ^{\frac{q_2}{q_1}}\frac{dt_2}{t_2}\right) ^{\frac1{q_2}}
	$$
	$$
	=\frac14\left( \int_0^\infty\left( \int_0^\infty\left( t_1^{-\theta_1/2}t_2^{-\theta_2/2}	\int_0^{t_2} \int_0^{t_1}\frac1{s_2} \int_0^{s_2}  f^{*_1*_2}(s_1,r_2)dr_2ds_1ds_2\right)^{q_1} \frac{dt_1}{t_1}\right) ^{\frac{q_2}{q_1}}\frac{dt_2}{t_2}\right) ^{\frac1{q_2}}
	$$
	$$
	\leq\frac1{\theta_1\theta_2}\left( \int_0^\infty\left( \int_0^\infty\left( t_1^{1/p_1}t_2^{1/p_2}\frac1{t_2} \int_0^{t_2}  f^{*_1*_2}(t_1,r_2)dr_2\right)^{q_1} \frac{dt_1}{t_1}\right) ^{\frac{q_2}{q_1}}\frac{dt_2}{t_2}\right) ^{\frac1{q_2}}
	$$
	$$
	\leq\frac{p_2}{\theta_1\theta_2}\|f\|_{L_{\bar p,\bar q}}\leq\frac2{\theta_1\theta_2}\|f\|_{L_{\bar p,\bar q}},
	$$
	$$
	I_{10}=\left( \int_0^\infty\left( \int_0^\infty\left( t_1^{1-\theta_1}t_2^{-\theta_2}	\int_0^{w_2} \left(  \int_{w_1}^\infty (f^{*_1*_2}(s_1,s_2))^2ds_1\right)^{\frac12} ds_2\right)^{q_1} \frac{dt_1}{t_1}\right) ^{\frac{q_2}{q_1}}\frac{dt_2}{t_2}\right) ^{\frac1{q_2}}
	$$
	$$
	=\frac14\left( \int_0^\infty\left( \int_0^\infty\left( t_1^{(1-\theta_1)/2}t_2^{-\theta_2/2}	\int_0^{t_2} \left(  \int_{t_1}^\infty (f^{*_1*_2}(s_1,s_2))^2ds_1\right)^{\frac12} ds_2\right)^{q_1} \frac{dt_1}{t_1}\right) ^{\frac{q_2}{q_1}}\frac{dt_2}{t_2}\right) ^{\frac1{q_2}}
	$$
	$$
	\leq \frac1{(1-\theta_1)^{\beta_1}\theta_2}\|f\|_{L_{\bar p,\bar q}},
	$$
Similarly
		$$
	I_{01}\leq\frac2{(1-\theta_2)^{\beta_2}\theta_1}\|f\|_{L_{\bar p,\bar q}},
	$$
		$$
	I_{11}\leq\frac1{(1-\theta_2)^{\beta_2}(1-\theta_1)^{\beta_1}}\|f\|_{L_{\bar p,\bar q}}.
	$$
	\end{proof}	
\section{On the Fourier coefficients of functions in the space $L_{\bar 2,\bar q}$}	
	\begin{lemma}\label{Le3}
Let $1 < q_1, q_2 < 2$, 
$\Phi = \{\varphi_{m_1}(x_1)\}_{m_1=1}^{\infty}$, 
$\Psi = \{\psi_{m_2}(x_2)\}_{m_2=1}^{\infty}$ 
be orthonormal systems on $[0,1]$ which are uniformly bounded, that is,
\[
\|\varphi_{m_1}\|_{\infty} \le M_1, \quad m_1 = 1,2,\ldots,
\]
\[
\|\psi_{m_2}\|_{\infty} \le M_2, \quad m_2 = 1,2,\ldots.
\]
then the following inequalities hold\footnote{The inequalities are understood in the following sense: the finiteness of the right-hand side implies the finiteness of the left-hand side, and the corresponding relation is valid.}:
		\begin{equation}
			\begin{aligned}
				\left(\sum\limits_{k_2=1}^{N_2}\sum\limits_{k_1=1}^{N_1} 
				(a_{k_1,k_2}^{*_1,*_2})^{2} \right)^\frac{1}{2}
				\leq M_2 N_2^{\tfrac{1}{2}}\|f\|_{\mathrm{L}_{(2,1)}}, \\
				\left(\sum\limits_{k_2=1}^{N_2}\sum\limits_{k_1=1}^{N_1} 
				(a_{k_1,k_2}^{*_1,*_2})^{2} \right)^\frac{1}{2}\leq M_1N_1^{\tfrac{1}{2}}\|f\|_{L_{(1,2)}} \\
				\left(\sum\limits_{k_2=1}^{N_2}\sum\limits_{k_1=1}^{N_1} 
				(a_{k_1,k_2}^{*_1,*_2})^{2} \right)^\frac{1}{2}\leq M_1M_2(N_1N_2)^{\tfrac{1}{2}}\|f\|_{L_{(1,1)}} \\
				\left(\sum\limits_{k_2=1}^{N_2}\sum\limits_{k_1=1}^{N_1} 
				(a_{k_1,k_2}^{*_1,*_2})^{2} \right)^\frac{1}{2}\leq \|f\|_{L_{(2,2)}}, 
			\end{aligned}
		\end{equation}
        where $L_{(p_1,p_2)}$ is the Lebesgue space with the mixed metric (see \eqref{4.3})
	\end{lemma}
	Notice that bounded sets $e \in \mathbb{Z}$: $|e|=N_1$ and for any $m_1 \in e$ exist $\omega(m_1)\subset \mathbb{Z}$ that $|\omega(k_1)|=N_2$ and \\
	\begin{equation*}
		\begin{aligned}
			\left(\sum\limits_{k_2=1}^{N_2}\sum\limits_{k_1=1}^{N_1} 
			(a_{k_1,k_2}^{*_1,*_2})^{2} \right)^{\frac12}=\left( \sum_{m_1\in e}\sum_{m_2 \in \omega(m_1)}|a_{m_1,m_2}|^2\right)^{\frac12} \\
			=\left(\sum_{m_1\in e}\sum_{m_2 \in \omega(m_1)}\left| \int_0^{1}\int_0^{1} f(x_1, x_2)\psi_{m_1}(x_1)dx_1\phi_{m_2}(x_2)dx_2 \right|^2\right)^\frac{1}{2}\\
			\leq M_2 \left(\sum_{m_1\in e}\sum_{m_2 \in \omega(m_1)}\left(\int_0^{1}\left| \int_0^{1} f(x_1, x_2)\psi_{m_1}(x_1)dx_1\right|dx_2\right)^2\right)^\frac{1}{2}\\
			=M_2 \,N_2^\frac{1}{2}\int_0^{1}\left(\sum_{m_i \in e}\left|\int_0^{1}f(x_1, x_2)\psi_{m_1}(x_1)dx_1\right|^2\right)^\frac{1}{2}dx_2\leq M_2\,N_2^\frac{1}{2}\|f\|_{L_{(2,1)}}.
		\end{aligned}
	\end{equation*}
	and
	\begin{equation*}
		\begin{aligned}
			\left(\sum_{m_1\in e}\sum_{m_2 \in \omega(m_1)}\left| \int_0^{1}\int_0^{1} f(x_1, x_2)\psi_{m_1}(x_1)dx_1\phi_{m_2}(x_2)dx_2 \right|^2\right)^\frac{1}{2}\\
			\left(\sum_{m_1\in e}\sum_{m_2 \in\Bbb Z}\left| \int_0^{1}\int_0^{1} f(x_1, x_2)\psi_{m_1}(x_1)dx_1\phi_{m_2}(x_2)dx_2 \right|^2\right)^\frac{1}{2}\\
			=\left(\sum_{m_1\in e}\left(\int_0^{1} \left|\int_0^{1}f(x_1, x_2)\psi_{m_1}(x_1)dx_1\right|^{2}\right)dx_2\right)^{\frac{1}{2}}\\
			\leq M_1 N^{\frac{1}{2}}\left(\int_0^{1}\left(\int_0^{1}|f(x_1, x_2)|dx_1\right)^2dx_2\right)^{\frac{1}{2}}=M_1 N_1^{\frac{1}{2}}\|f\|_{L_{(1,2)}}
		\end{aligned}
	\end{equation*} 
	It is clear that
	$$
	\left(\sum_{m_1\in e}\sum_{m_2 \in \omega(m_1)}\left| \int_0^{1}\int_0^{1} f(x_1, x_2)\psi_{m_1}(x_1)dx_1\phi_{m_2}(x_2)dx_2 \right|^2\right)^\frac{1}{2}\leq M_1 M_2|e|^\frac{1}{2}|\omega_k|^\frac{1}{2}\|f\|_{L_{(1,1)}}
	$$
	$$
	=M_1M_2(N_1N_1)^\frac12 \|f\|_{L_{(1,1)}}
	$$
	$$
	\left(\sum_{m_1\in e}\sum_{m_2 \in \omega(m_1)}\left| \int_0^{1}\int_0^{1} f(x_1, x_2)\psi_{m_1}(x_1)dx_1\phi_{m_2}(x_2)dx_2 \right|^2\right)^\frac{1}{2}\leq\|f\|_{L_{(2,2)}}
	$$ follows from Parseval's inequality. 
	\vskip0.5cm
	\begin{lemma}\label{Le2}
		Let $0<\overline{q}=(q_1,q_2)\leq\infty, \overline{\theta}=(\theta_1,\theta_2)\in(0,1)^2$ and $f\in (L_{(1,1)}[0,1]^2, L_{{2,2}}[0,1]^2))_{\bar \theta,\bar q}$. let $f \sim \sum_{k_2 \in \mathbb{N} }\sum_{ k_1 \in \mathbb{N}} a_{k_1,k_2}[0,1]^2 \phi_{m_1}(x_1)\psi_{m_2}(x_2)$ then the following inequality holds: 
		\begin{gather}
			\left(\sum_{k_2 =0}^\infty\left(\sum_{k_1=0}^\infty\left( 2^{\frac{k_1}{p'_1}+\frac{k_2}{p'_2}}\left[\frac{1}{2^{k_1+k_2}} \sum_{m_2=1}^{2^{k_2}} \sum_{m_1=1}^{2^{k_1}}\left(a_{m_1 m_2}^{*_2, *_1}(f)\right)^2\right]^{\frac{1}{2}}\right)^{q_1} \right)^{\frac{q_2}{q_1}}\right)^{\frac{1}{q_2}}\leq c\|f\|_{(L_{(1,1)}, L_{{(2,2)}})_{\bar \theta,\bar q}},
		\end{gather}
		where $\theta_1, \theta_2$ and $\frac{1}{p_i}=1-\frac{\theta_1}{2} \quad i=1,2$ are  the corresponding constant does not depend on the parameters.
	\end{lemma}
	\begin{proof}
	Let $f=f_{00}+f_{10}+f_{01}+f_{11}$ be an arbitrary decomposition of the function $f$, where $f_{00}\in L_{(1,1)}(\mathbb{T}^2),f_{10}\in L_{(2,1)}(\mathbb{T}^2),f_{01}\in L_{(1,2)}(\mathbb{T}^2), f_{11}\in L_{(2,2)}(\mathbb{T}^2)$, then
		$$
		\left(\sum_{m_2=1}^{2^{k_2}}\sum_{m_1=1}^{2^{k_1}}\left(a_{m_1, m_2}^{*_1, *_2}(f)\right)^2\right)^{1 /2}
	$$
	$$
	\leq2\left( \left(\sum_{m_2=1}^{2^{k_2-1}}\sum_{m_1=1}^{2^{k_1-1}}\left(a_{m_1, m_2}^{*_2, *_1}\left(f_{00}\right)\right)^2\right)^{1 /2}+\left(\sum_{m_2=1}^{2^{k_2-1}}\sum_{m_1=1}^{2^{k_1-1}}\left(a_{m_1, m_2}^{*_2, *_1}\left(f_{10}\right)\right)^2\right)^{1 /2}\right. 
	$$
	$$
	\left. +\left(\sum_{1}^{2^{k_2-1}}\sum_{1}^{2^{k_1-1}}\left(a_{m_1, m_2}^{*_2, *_1}\left(f_{01}\right)\right)^2\right)^{1 /2}+\left(\sum_{1}^{2^{k_2-1}}\sum_{1}^{2^{k_1-1}}\left(a_{m_1, m_2}^{*_2, *_1}\left(f_{11}\right)\right)^2\right)^{1 /2}\right) 
	$$
Next, applying Lemma~\ref{Le3}, we obtain
	$$
	\lesssim 2^{\frac{k_1+k_2}2}\|f_{00}\|_{L_{(1,1)}} +2^{\frac{k_1}2}\|f_{01}\|_{L_{(1,2)}} +2^{\frac{k_2}2}\|f_{10}\|_{L_{(2,1)}} +\|f_{11}\|_{L_{(2,2)}}
	$$
	
Taking into account the arbitrariness of the decomposition   $f=f_{00}+f_{10}+f_{01}+f_{11}$  we have
	$$
	\left(\sum_{m_2=1}^{2^{k_2}}\sum_{m_1=1}^{2^{k_1}}\left(a_{m_1, m_2}^{*_2, *_1}(f)\right)^2\right)^{1 /2}
	\le K\left(2^{\frac{k_1}2},2^{\frac{k_2}2};f\;;L_{(2,2)}, L_{(1,1)} \right) 
	$$
Therefore we obtain
	$$
	\left( \sum_{k_2=0}^\infty\left( \sum_{k_1=0}^\infty\left(2^{k_1(\frac1{p'_1}-\frac12)+(\frac1{p'_2}-\frac12)}
	K\left(2^{\frac{k_1}2},2^{\frac{k_2}2};f\;;L_{(2,2)}, L_{(1,1)} \right)\right) ^{q_1 } \right) ^{\frac{q_2}{q_1}}\right) ^{1/q_2}
	$$
	$$
	=\left( \sum_{k_2=0}^\infty\left( \sum_{k_1=0}^\infty\left(2^{-\frac{(1-\theta_1)}2-\frac{(1-\theta_2)}2}
	K\left(2^{\frac{k_1}2},2^{\frac{k_2}2};f\;;L_{(2,2)}, L_{(1,1)} \right)\right) ^{q_1 } \right) ^{\frac{q_2}{q_1}}\right) ^{1/q_2}
	$$
	$$
	\asymp \left( \int_1^\infty\left( \int_1^\infty\left(t_1^{-(1-\theta_1)}t_2^{-(1-\theta_2)}
	K\left(t_1,t_2;f\;;L_{(2,2)}, L_{(1,1)} \right)\right) ^{q_1 }\frac{dt_1}{t_1} \right) ^{\frac{q_2}{q_1}}\frac{dt_2}{t_2}\right) ^{1/q_2}
	$$
	$$
	\leq \|f\|_{(L_{(2,2)}, L_{{(1,1)}})_{\overline{1- \theta},\bar q}}= \|f\|_{(L_{(1,1)}, L_{{(2,2)}})_{\bar \theta,\bar q}}
	$$
Here we have used Lemma \ref{Le0}. 
\end{proof}
		\begin{theorem}\label{Te3}
		Let $0<\overline{q}=(q_1,q_2)\leq\infty, \overline{\theta}=(\theta_1,\theta_2)\in(0,1)^2$, $\beta_i=\max\{\frac12, \frac1{q_i}\},\;\frac{1}{p_i}=1-\frac{\theta_1}{2} \quad i=1,2$ and $f\in L_{\bar p,\bar q}$. let $f \sim \sum_{k_2 \in \mathbb{N} }\sum_{ k_1 \in \mathbb{N}} a_{k_1,k_2} \phi_{m_1}(x_1)\psi_{m_2}(x_2)$ then the following inequality holds: 
		\begin{gather}\label{e9}
			\left(\sum_{k_2 =0}^\infty\left(\sum_{k_1 =0}^\infty \left( 2^{\frac{k_1}{p'_1}+\frac{k_2}{p'_2}}\left[\frac{1}{2^{k_1+k_2}} \sum_{m_2=1}^{2^{k_2}} \sum_{m_1=1}^{2^{k_1}}\left(a_{m_1 m_2}^{*_2, *_1}(f)\right)^2\right]^{\frac{1}{2}}\right)^{q_1} \right)^{\frac{q_2}{q_1}}\right)^{\frac{1}{q_2}}\leq C(\theta_1,\theta_2, \bar\beta)\|f\|_{L_{\bar p,\bar q}[0,1]^2},
		\end{gather}
	where the constant $C(\theta_1,\theta_2,\bar \beta)$, independent of $\theta_1$ and $\theta_2$, satisfies relation~\eqref{const}.
	\end{theorem}	
\begin{proof}
The statement follows immediately from Lemma~\ref{Le2} and Theorem~\ref{Te1}.
	\end{proof}
	\begin{theorem}\label{Te4} Let $0<\bar q \leq\infty, \;\beta_i=\max\{\frac12, \frac1{q_i}\}, \; \theta_i=\lambda _i-\beta _i\geq 0, \; i=1,2\}$ and let $f \sim \sum_{k_2 \in \mathbb{N} }\sum_{ k_1 \in \mathbb{N}} a_{k_1,k_2} \phi_{m_1}(x_1)\psi_{m_2}(x_2)$. If $f\in G^{\overline\theta}L_{\overline2,\overline q}$, то  $a\in G^{\overline\lambda}\ell^{*}_{\overline2,\overline q} $and the inequality
		\[
		\|a\|_{G^{\bar\lambda}\ell^{*}_{\bar2,\bar q}} \lesssim \|f\|_{G^{\bar \theta}L_{\bar 2,\bar  q}}
		\]
holds. In particular,
		\[
		\|a\|_{G^{\bar\beta}\ell^{*}_{\bar2,\bar q}} \lesssim \|f\|_{L_{\bar 2,\bar  q}}.
		\]
	\end{theorem}
	\begin{proof}
		Let \( 0<\varepsilon_1,\varepsilon_2 <1 \). Assume that  $\frac{1}{p_1}=\frac{1}{2}+\varepsilon_1$,$\frac{1}{p_2}=\frac{1}{2}+\varepsilon_2$, then from Theorem \ref{Te3} (see \eqref{e9} ,\eqref{const}) we have 
		$$	\left(\sum_{k_2=0}^\infty\left(\sum_{k_1=0}^\infty\left( 2^{-\varepsilon_1{k_1}-\varepsilon_2{k_2}}\left[ \sum_{m_2=1}^{2^{k_2}} \sum_{m_1=1}^{2^{k_1}}\left(a_{m_1 m_2}^{*_2, *_1}(f)\right)^2\right]^{\frac{1}{2}}\right)^{q_1} \right)^{\frac{q_2}{q_1}}\right)^{\frac{1}{q_2}},
		$$
		$$
			\leq  c\varepsilon_1^{-\beta_1} \varepsilon_2^{-\beta_2}  
			\left( \int_0^1\left( \int_0^1 (t_2^{\frac{1}{2} + \varepsilon_2}t_1^{\frac{1}{2} + \varepsilon_1} f^{*_1,*_2}(t_1,t_2))^{q_1} \frac{dt_1}{t_1} \right)^{\frac{q_2}{q_1}}\frac{dt_2}{t_2} \right)^{\frac{1}{q_2}}.    
	$$
		
		Now multiplying both sides by $\varepsilon_1^{\lambda_1}\varepsilon_2^{\lambda_2}$, we get  
		$$
			\varepsilon_1^{\beta_1}\varepsilon_2^{\beta_2}
			\left(\sum_{k_2 =0}^\infty\left(\sum_{k_1=0}^\infty\infty\left( 2^{-\varepsilon_1{k_1}-\varepsilon_2{k_2}}\left[ \sum_{m_2=1}^{2^{k_2}} \sum_{m_1=1}^{2^{k_1}}\left(a_{m_1 m_2}^{*_2, *_1}(f)\right)^2\right]^{\frac{1}{2}}\right)^{q_1} \right)^{\frac{q_2}{q_1}}\right)^{\frac{1}{q_2}}
			$$
			$$
						\leq   c\varepsilon_1^{\theta_1} \varepsilon_2^{\theta_2}  
			\left( \int_0^1\left( \int_0^1 (t_2^{\frac{1}{2} + \varepsilon_2}t_1^{\frac{1}{2} + \varepsilon_1} f^{*_1,*_2}(t_1,t_2))^{q_1} \frac{dt_1}{t_1} \right)^{\frac{q_2}{q_1}}\frac{dt_2}{t_2} \right)^{\frac{1}{q_2}}
			\leq c \|f\|_{G^{\overline\theta}L_{\bar2,\bar q}}.
	$$  
    Taking into account the arbitrariness of the parameter $\overline{\varepsilon} = \varepsilon_1, \varepsilon_2$ we obtain the required  inequality.
	\end{proof}
\begin{theorem}
Let $2\leq \bar q\leq\infty$ and $f\in L_{\bar2,\bar\tau}$. If $f\in L_{\bar p,\bar q}[0,1]^2$ then the following inequality holds: 
$$\sup_{k_2,k_1\in \Bbb N}\frac{1}{(\ln k_1)^{\frac{1}{2}-\frac{1}{q_1}}(\ln k_2)^{\frac{1}{2}-\frac{1}{q_2}}}\left(\sum_{m_2=1}^{k_2} \sum_{m_1=1}^{k_1}\left(a_{m_1 m_2}^{*_2, *_1}(f)\right)^2\right)^\frac{1}{2}\lesssim \|f\|_{L_{\bar2,\bar\tau}},$$ where $\bar k=\max(k,2)$
\end{theorem}
\begin{proof}	
We begin by proving the following inequality:
\begin{equation}\label{5.4*}
\sup_{n_2,n_1\in \Bbb N} n_2^{1/q_2-1/2} n_1^{1/q_1-1/2}	\left[ \sum_{m_2=1}^{2^{n_2}} \sum_{m_1=1}^{2^{n_1}}\left(a_{m_1 m_2}^{*_2, *_1}(f)\right)^2\right]^{\frac{1}{2}} \lesssim \|f\|_{L_{\bar 2,\bar  q}}.
\end{equation}
		The statement follows from Theorem~\ref{Te4}. Indeed, let $n_1 ,n_2 \in \Bbb N$	
	\[
	\|a\|_{G^{1/2}\ell^{*}_{\bar2,\bar q}} \geq \sup_{{}_{0<\varepsilon_1\leq 1}^{0<\varepsilon_2\leq 1}}
	\varepsilon_1^{1/2}\varepsilon_2^{1/2}
	\left(\sum_{k_2 =n_2}^\infty\left(\sum_{k_1 =n_1}^\infty\left( 2^{-\varepsilon_1{k_1}-\varepsilon_2{k_2}}\left[ \sum_{m_2=1}^{2^{k_2}} \sum_{m_1=1}^{2^{k_1}}\left(a_{m_1 m_2}^{*_2, *_1}(f)\right)^2\right]^{\frac{1}{2}}\right)^{q_1} \right)^{\frac{q_2}{q_1}}\right)^{\frac{1}{q_2}}
	\]
	$$
	\geq \left[ \sum_{m_2=1}^{2^{n_2}} \sum_{m_1=1}^{2^{n_1}}\left(a_{m_1 m_2}^{*_2, *_1}(f)\right)^2\right]^{\frac{1}{2}} \sup_{{}_{0<\varepsilon_1\leq 1}^{0<\varepsilon_2\leq 1}}
	\varepsilon_1^{1/2}\varepsilon_2^{1/2}    	\left(\sum_{k_2 =n_2}^\infty2^{-\varepsilon_2 q_2{k_2}}\right) ^{1/q_2} 	\left(\sum_{k_1 =n_1}^\infty2^{-\varepsilon_1 q_1{k_1}}\right) ^{1/q_1}            
$$	
$$
\asymp \left[ \sum_{m_2=1}^{2^{n_2}} \sum_{m_1=1}^{2^{n_1}}\left(a_{m_1 m_2}^{*_2, *_1}(f)\right)^2\right]^{\frac{1}{2}} \sup_{{}_{0<\varepsilon_1\leq 1}^{0<\varepsilon_2\leq 1}}
\varepsilon_1^{1/2-1/q_1}\varepsilon_2^{1/2-1/q_2}    	2^{-\varepsilon_2 n_2} \;2^{-\varepsilon_1 {n_1}}             
$$
	$$
\asymp  n_2^{1/q_2-1/2} n_1^{1/q_1-1/2} \left[ \sum_{m_2=1}^{2^{n_2}} \sum_{m_1=1}^{2^{n_1}}\left(a_{m_1 m_2}^{*_2, *_1}(f)\right)^2\right]^{\frac{1}{2}} .
	$$
Taking into account the arbitrariness of the choice of \(n_{1}\) and \(n_{2}\), 
and applying we obtain the desired conclusion.\\
Let $k_1,k_2\in \mathbb{N}, k_i\geq2$, then there exist  $n_1,n_2\in \mathbb{N}$ such that $2^{n_1-1}\leq k_1\leq 2^{k_i}, i=1,2$. Therefore:
	\begin{equation}\label{5.5}
			\begin{aligned}
\frac{1}{(\ln k_1)^{\frac{1}{2}-\frac{1}{q_1}}(\ln k_2)^{\frac{1}{2}-\frac{1}{q_2}}}
\left(\sum_{m_2=1}^{k_2}\sum_{m_1=1}^{k_1}
(a_{m_1 m_2}^{*_2, *_1}(f))^2\right)^{\frac12}\\
\leq 
\frac{1}{n_1^{\frac{1}{2}-\frac{1}{q_1}}n_2^{\frac{1}{2}-\frac{1}{q_2}}}
\left(\sum_{m_2=1}^{2k_2+1}\sum_{m_1=1}^{2k_1+1}
(a_{m_1 m_2}^{*_2, *_1}(f))^2\right)^{\frac12} \\
\leq 4 \sup_{n_2,n_1\in\mathbb{N}}
\frac{1}{n_1^{\frac{1}{2}-\frac{1}{q_1}}n_2^{\frac{1}{2}-\frac{1}{q_2}}}
\left(\sum_{m_2=1}^{2k_2}\sum_{m_1=1}^{2k_1}
(a_{m_1 m_2}^{*_2, *_1}(f))^2\right)^{\frac12},
			\end{aligned}
		\end{equation}
        Thus, $\eqref{5.5}$ follows from $\eqref{5.4*}$. 

        \end{proof}  
\textbf{Acknowledgements} The authors thank Professors Erlan Nursultanov and  Muhammad Asad Zaighum, Arash Ghorbanalizadeh for their invaluable guidance, support, and insightful discussions.\\ 
\textbf{Funding} This research is funded by the Science Committee of MSHE of Kazakhstan (Grant No. AP23488613). \\

\end{document}